\documentclass[12pt,reqno]{amsart}

\usepackage[T1]{fontenc}
\usepackage[utf8]{inputenc}
\usepackage[english]{babel}
\usepackage{amsthm,amssymb,amsmath}
\usepackage{hyperref}
\usepackage[top=1in, bottom=1in, left=1in, right=1in]{geometry}

\newtheorem{Theorem}{Theorem}[section]
\newtheorem{Lemma}[Theorem]{Lemma}
\newtheorem{Corollary}[Theorem]{Corollary}

\newtheorem{athm}{Theorem}

\theoremstyle{definition}

\newtheorem{Example}{Example}

\theoremstyle{remark}
\newtheorem{Note}{Remark}

\DeclareMathOperator{\Aut}{Aut}

\title{Finite skew braces whose additive group is a Z-group}

\author[Marco Damele]{Marco Damele}
\thanks{Dipartimento di Matematica, Università di Cagliari, Via Ospedale 72, 09124 Cagliari, Italy; \texttt{marco.damele@unica.it}; ORCID 0009-0008-3088-5766}

\date{}

\begin{document}
\maketitle

\begin{abstract}

Rump proved in \cite[Theorem~1]{Rump2018ClassificationOC} that if a finite skew brace has cyclic additive group, then its multiplicative group is solvable and almost Sylow cyclic. In this paper we show that this rigidity persists when the additive group is a \(Z\)-group. More precisely, we prove that if \(B\) is a finite skew brace whose additive group is a \(Z\)-group, then \((B,\cdot)\) is solvable and almost Sylow cyclic. In addition, we show that every such skew brace is supersolvable; in particular, \((B,\cdot)\) is \(2\)-nilpotent. This extends \cite[Theorem~3.8]{ballesterbolinches2024finiteskewbracessquarefree} and recovers, in this broader setting, another result of Rump \cite[Proposition 13]{Rump2018ClassificationOC}. Finally, we prove that for skew braces of odd order the additive group is a \(Z\)-group if and only if the multiplicative group is a \(Z\)-group.

\end{abstract}

\medskip



\section{Introduction}
A \emph{skew brace} is a triple \((B,+,\cdot)\), where \((B,+)\) and \((B,\cdot)\) are groups on the same underlying set \(B\), called the \emph{additive group} and the \emph{multiplicative group} of \(B\), respectively, and satisfying
\begin{equation} \label{eq:left-dist}
    a\cdot (b+c)=a\cdot b-a+a\cdot c
\qquad
\text{for all } a,b,c\in B.
\end{equation}
Introduced by Guarnieri and Vendramin \cite{Guarnieri_2016}, skew braces provide a natural algebraic framework for the study of set-theoretic solutions of the Yang--Baxter equation. A natural problem in the theory of skew braces is to understand the interplay between the additive group \((B,+)\) and the multiplicative group \((B,\cdot)\). In particular, one would like to determine how properties of the additive group constrain those of the multiplicative group, and vice versa. More broadly, one seeks to understand how the properties of these two groups influence the internal structure of the skew brace \(B\) itself. For example, a classical result due to Smoktunowicz and Vendramin states that if \(B\) is a finite skew brace and \((B,+)\) is nilpotent, then \((B,\cdot)\) is solvable \cite{Smoktunowicz_2018}. This naturally leads to the broader question of whether solvability of the additive group forces solvability of the multiplicative group in every finite skew brace. At present, this problem remains open and continues to attract considerable attention; see \cite{Byott2, GorshkovNasybullov2021, Nas2, Tsang_2020}. By contrast, the corresponding statement fails in the infinite setting, as shown by examples due to Nasybullov \cite{Nas2}. On the other hand, in the Lie setting the analogue has recently been settled affirmatively in \cite{dameleLoi}. The case where the additive group is cyclic was studied by Rump \cite[Theorem~1]{Rump2018ClassificationOC}, who proved that if \((B,+)\) is cyclic, then \((B,\cdot)\) is solvable and almost Sylow cyclic. This provides a particularly striking instance of the principle that strong restrictions on the additive group may force a rigid structure on the multiplicative group. The main goal of this paper is to extend this phenomenon from cyclic groups to \(Z\)-groups. This is a natural direction in which to proceed, since cyclic groups are precisely the nilpotent finite \(Z\)-groups: a finite group is cyclic if and only if it is both nilpotent and a \(Z\)-group. Thus \(Z\)-groups provide the natural broader setting in which one may ask whether Rump's conclusion continues to hold. Our first result shows that this is indeed the case.

\begin{athm} \label{th: generalizedRump}
Let \(B\) be a finite skew brace such that \((B,+)\) is a \(Z\)-group. Then \((B,\cdot)\) is solvable and almost Sylow cyclic.
\end{athm}

Theorem~\ref{th: generalizedRump} already shows that the assumption that the additive group is a \(Z\)-group imposes strong restrictions on the multiplicative group. In fact, this hypothesis has consequences at the level of the entire skew brace. Our second main result is perhaps the most striking aspect of the paper: from a condition on the additive group, one obtains a global structural property of the skew brace itself.

\begin{athm} \label{Th: B ssolv}
Let \(B\) be a finite skew brace such that \((B,+)\) is a \(Z\)-group. Then \(B\) is supersolvable.
\end{athm}

Theorem~\ref{Th: B ssolv} extends \cite[Theorem~3.8]{ballesterbolinches2024finiteskewbracessquarefree} and recovers, in the broader setting of skew braces whose additive group is a \(Z\)-group, \cite[Proposition~13]{Rump2018ClassificationOC}. It is natural to ask to what extent the assumption that one of the two groups is a \(Z\)-group forces the same conclusion for the other. In general, the converse of Theorem~\ref{th: generalizedRump} does not hold (see Example \ref{example1}). However, the situation becomes symmetric for skew braces of odd order, where the \(Z\)-group condition turns out to be equivalent for the additive and multiplicative groups.

\begin{athm} \label{Th: Multiplicative-Z-group}
Let \(B\) be a finite skew brace of odd order. Then \((B,+)\) is a \(Z\)-group if and only if \((B,\cdot)\) is a \(Z\)-group.
\end{athm}

The proofs of these results rely on the realization of skew braces in terms of regular subgroups of holomorphs, together with standard structural properties of finite \(Z\)-groups and supersolvable groups. For the convenience of the reader, Section~\ref{Preliminaries} collects the necessary background on skew braces, holomorphs, and \(Z\)-groups, while Section~\ref{Proofs} contains the proofs of the main results.
\section{Preliminaries} \label{Preliminaries}

\subsection{Skew braces}
Let \(B\) be a skew brace. Equation~\eqref{eq:left-dist} yields the same identity element in both additive and multiplicative groups. We denote it by~$0$. Then, the multiplicative group acts on the additive group by means of a homomorphism $\lambda\colon a \in (B,\cdot) \mapsto \lambda_a\in \Aut(B,+)$, with $\lambda_a(b) = -a+ab$ for every $a,b\in B$. A subset \(S\subseteq B\) is a \emph{subbrace} of \(B\) if \((S,+)\leq (B,+)\) and \((S,\cdot)\leq (B,\cdot)\); in this case \((S,+,\cdot)\) is itself a skew brace. In particular, if \(H\) is a characteristic subgroup of \((B,+)\), then \((H,\cdot)\leq (B,\cdot)\), so \((H,+,\cdot)\) is a subbrace of \(B\) (see \cite[Lemma~2.6]{damele2026finitegroupspropercharacteristic}). A subset \(I\subseteq B\) is an \emph{ideal} of \(B\) if $(I,+)\trianglelefteq (B,+)$, $ (I,\cdot)\trianglelefteq (B,\cdot),$ and $\lambda_b(I)=I$ for every \(b\in B\).  In this case one can form the quotient skew brace \((B/I,+,\cdot)\), where additive and multiplicative cosets coincide, that is, $a+I=aI$ for every $a \in B$. The \emph{socle} of \(B\) is defined by $\mathrm{Soc}(B)=\ker\lambda\cap Z(B,+)$. It is well known that \(\mathrm{Soc}(B)\) is an ideal of \(B\). Following \cite{ballesterbolinches2024finiteskewbracessquarefree}, we say that \(B\) is \emph{supersolvable} if there exists a chain of ideals $\{0\}=I_0\leq I_1\leq \cdots \leq I_n=B$ such that each factor \(I_{j+1}/I_j\) has prime order.
\subsection{Holomorph and skew braces}

Let \(H\) be a finite group. Recall that its holomorph is $\mathrm{Hol}(H)=H\rtimes \mathrm{Aut}(H)$ where $\mathrm{Aut}(H)$ act on $H$ by evaluation. Every element of \(\mathrm{Hol}(H)\) can be written as a pair \((x,f)\), with $x \in H$ and $f \in \mathrm{Aut}(H)$ and \(\mathrm{Hol}(H)\) acts naturally on \(H\) via
\[
(x,f)\cdot h=x\,f(h)
\qquad
\text{for all } h\in H.
\]
Now let \(G\leq \mathrm{Hol}(H)\). By restriction, \(G\) acts on \(H\) through the same rule. We say that \(G\) is a \emph{regular subgroup} of \(\mathrm{Hol}(H)\) if this action is regular, that is, if it is both transitive and free. Equivalently, \(G\) is regular if for every \(h_1,h_2\in H\) there exists a unique \(g\in G\) such that
\[
g\cdot h_1=h_2.
\]

\begin{Note} \label{RmK: transitive of the same dimension}
Let \(G\leq \mathrm{Hol}(H)\), and suppose that $|G|=|H|$. Then \(G\) is regular if and only if it acts transitively on \(H\). Indeed, regularity always implies transitivity. Conversely, assume that the action of \(G\) on \(H\) is transitive, and let \(h\in H\). Thus $|G|=|\mathrm{Orb}_G(h)|$ where $\mathrm{Orb}_G(h)$ is the orbit of $h$. Let $G_h$ be the stabilizer of $h$. By the orbit-stabilizer theorem, $|G|=|\mathrm{Orb}_G(h)|\,|G_h|$. Since the action is transitive, we have $|\mathrm{Orb}_G(h)|=|H|=|G|$. Thus $|G_h|=1$. Since \(h\) was arbitrary, the action is free. Hence the action is regular.
\end{Note}

The connection between skew braces and regular subgroups of the holomorph is summarized in the following well-known theorem.

\begin{Theorem}[{\cite[Theorem~4.2]{Guarnieri_2016}}] \label{Connection}
Let \(H\) be a group and let \(G\) be a regular subgroup of \(\mathrm{Hol}(H)\). Then there exists a skew brace \((B,+,\cdot)\) such that $(B,+)\cong H$ and $(B,\cdot)\cong G$. Conversely, if \((B,+,\cdot)\) is a skew brace, then \((B,\cdot)\) can be identified with a regular subgroup of \(\mathrm{Hol}(B,+)\) via the map $(B,\cdot)\longrightarrow \mathrm{Hol}(B,+),
 \ b\longmapsto (b,\lambda_b)$.
\end{Theorem}

\subsection{Z-groups and almost Sylow cyclic groups}
Recall that a finite group \(G\) is called a \(Z\)-group if all its Sylow subgroups are cyclic. The next theorem gives a full description of such groups:

\begin{Theorem}[{\cite[Theorem~9.4.3]{Hall1959}}] \label{Zgroups}
Let \(G\) be a finite \(Z\)-group. Then \(G\) admits a presentation of the form
\[
G=\langle a,b \mid a^{m}=b^{n}=1,\; bab^{-1}=a^{r}\rangle,
\]
where
\[
(m,n)=(m,r-1)=1
\qquad\text{and}\qquad
r^{n}\equiv 1 \pmod{m}.
\]
\end{Theorem}

In particular, the commutator subgroup of a finite $Z$-group is cyclic. We also recall that a finite group \(G\) is said to be \emph{almost Sylow cyclic} if all Sylow subgroups of odd order are cyclic and either the Sylow \(2\)-subgroups of \(G\) are trivial or each of them contains a cyclic subgroup of index \(2\).

\section{Proofs of the main results} \label{Proofs}

\subsection{Proof of Theorem~\ref{th: generalizedRump}}

Before proving Theorem \ref{th: generalizedRump}, we recall a standard result on transitive group actions. For the reader's convenience, we include a proof.

\begin{Lemma}
Let \(G\) be a group acting transitively on a set \(\Omega\) with $|\Omega|=q^r$ where \(q\) is a prime number. If \(Q\in \operatorname{Syl}_q(G)\), then \(Q\) acts transitively on \(\Omega\).
\end{Lemma}

\begin{proof}
Fix an element \(\omega\in \Omega\), and let \(H=G_\omega\) be the stabilizer of \(\omega\) in \(G\). Since \(G\) acts transitively on \(\Omega\), the orbit--stabilizer theorem gives
$[G:H]=|\Omega|=q^r$. Write $|G|=q^n m$ where $q\nmid m$. Then $|H|=q^{\,n-r}m$. So the highest power of \(q\) dividing \(|H|\) is \(q^{\,n-r}\). Since \(Q\in \operatorname{Syl}_q(G)\), we have \(|Q|=q^n\). Moreover $|Q:Q\cap H| = |QH:H|$ and \(QH/H\) is a \(q\)-subgroup of \(G/H\).  Since $|G/H|=q^r$ it follows that $|QH:H| \le q^r$. Hence
\[
|Q\cap H|
= \frac{|Q|}{|Q:Q\cap H|}
\ge \frac{q^n}{q^r}
= q^{\,n-r}.
\]
But \(Q\cap H\le H\), and the largest power of \(q\) dividing \(|H|\) is exactly \(q^{\,n-r}\). Therefore $|Q\cap H|=q^{\,n-r}$. Now the stabilizer of \(\omega\) for the action of \(Q\) is $Q_\omega = Q\cap G_\omega = Q\cap H.$
So by the orbit--stabilizer theorem again,
\[
|\operatorname{Orb}_Q(\omega)|
= [Q:Q_\omega]
= [Q:Q\cap H]
= \frac{q^n}{q^{\,n-r}}
= q^r.
\]
Since \(|\Omega|=q^r\), the orbit $\operatorname{Orb}_Q(\omega)$ has the same size as \(\Omega\), and therefore
$\operatorname{Orb}_Q(\omega)=\Omega$. Thus \(Q\) acts transitively on \(\Omega\).
\end{proof}

We are now ready to prove Theorem \ref{th: generalizedRump}.

\begin{proof}[Proof of Theorem~\ref{th: generalizedRump}]
Assume, for a contradiction, that the statement is false, and let \(B\) be a minimal counterexample. Since \((B,+)\) is a \(Z\)-group, by Theorem \ref{Zgroups} its commutator subgroup \((B,+)'\) is cyclic. Hence, by \cite[Theorem 5.2]{li2025comparisonadditionmultiplicationskew}, the multiplicative group \((B,\cdot)\) is solvable. Because \((B,+)'\) is cyclic, by \cite[Lemma 2.17]{Alshorm2022} $(B,+)$ is supersolvable. Let \(q\) be the smallest prime dividing \(|B|\), and let \(Q\) be a Sylow \(q\)-subgroup of \((B,+)\). By \cite[Theorem 9.1]{Huppert1967FiniteGroupsI} \((B,+)\) admits a normal $q'$-Hall say \(H\), so that
\[
(B,+)=H\rtimes Q.
\]
 Since $H$ is a normal $q'$-Hall, by \cite[Main Theorem 1.7]{Huppert1967FiniteGroupsI} $H$ is characteristic in $(B,+)$. Thus \((H,+,\cdot)\) is a subbrace of $B$. Since \((H,+)\) is again a \(Z\)-group, minimality of \(B\) implies that \((H,\cdot)\) is almost Sylow cyclic. Now let \(p\) be a prime dividing \(|B|\), and let \(P\in \mathrm{Syl}_p(B,\cdot)\). If \(p\neq q\), then \(p\mid |H|\) and $p$ is odd. Moreover, every Sylow \(p\)-subgroup of \((H,\cdot)\) is also a Sylow \(p\)-subgroup of \((B,\cdot)\). Hence \(P\) is cyclic. It remains to deal with the prime \(q\). Let \(\widetilde{Q}\in \mathrm{Syl}_q(B,\cdot)\). We claim that \(\widetilde{Q}\) is cyclic if \(q\) is odd, and that \(\widetilde{Q}\) contains a cyclic subgroup of index at most \(2\) if \(q=2\). Since \((B,\cdot)\) is a regular subgroup of \(\mathrm{Hol}(B,+)\), its natural action on the quotient \((B,+)/H\) yields a homomorphism
\[
\pi\colon (B,\cdot)\longrightarrow \mathrm{Hol}\bigl((B,+)/H\bigr).
\]
The image \(\pi(B,\cdot)\) acts transitively on \((B,+)/H\). Indeed, if \(x+H,y+H\in (B,+)/H\), then the regularity of the action of \((B,\cdot)\) on \((B,+)\) ensures that there exists \(g\in (B,\cdot)\) such that \(g\cdot x=y\). Passing to cosets, we obtain
\[
\pi(g)\cdot (x+H)=y+H.
\]
Set
\[
\Omega:=(B,+)/H,
\qquad |\Omega|=[(B,+):H]=q^r.
\]
Since \(\pi(B,\cdot)\) is transitive on \(\Omega\), and \(\pi(\widetilde{Q})\) is a Sylow \(q\)-subgroup of \(\pi(B,\cdot)\), it follows that \(\pi(\widetilde{Q})\) is transitive on \(\Omega\). On the other hand, \(\pi(\widetilde{Q})\) is a homomorphic image of \(\widetilde{Q}\), so
\[
|\pi(\widetilde{Q})|\leq |\widetilde{Q}|=q^r.
\]
But \(\pi(\widetilde{Q})\) is transitive on a set of size \(q^r\), and hence
\[
|\pi(\widetilde{Q})|\geq q^r.
\]
Thus
\[
|\pi(\widetilde{Q})|=q^r.
\]
It follows by Remark \ref{RmK: transitive of the same dimension} that \(\pi(\widetilde{Q})\) acts regularly on \(\Omega\), and the restriction
\[
\pi|_{\widetilde{Q}}\colon \widetilde{Q}\longrightarrow \pi(\widetilde{Q})
\]
is an isomorphism, since it is surjective and \(|\widetilde{Q}|=|\pi(\widetilde{Q})|\). Consequently, \(\widetilde{Q}\) is isomorphic to a regular subgroup of
\[
\mathrm{Hol}\bigl((B,+)/H\bigr).
\]
Since \((B,+)/H\cong Q \cong  C_{q^r}\) is cyclic, we may identify
\[
\mathrm{Hol}\bigl((B,+)/H\bigr)\cong \mathrm{Hol}(C_{q^r}).
\]
Hence, using Theorem \ref{Connection} and \cite[Theorem 1]{Rump2018ClassificationOC}, the group \(\widetilde{Q}\) is cyclic if \(q\) is odd, while for \(q=2\) it contains a cyclic subgroup of index at most \(2\). Therefore \((B,\cdot)\) is almost Sylow cyclic, contrary to the choice of \(B\). This completes the proof.
\end{proof}

The next example shows that the converse of Theorem~\ref{th: generalizedRump} fails. 

\begin{Example} \label{example1}
There exists a non-trivial skew brace \(B\) such that \((B,\cdot)\) is almost Sylow cyclic, while \((B,+)\) is not a \(Z\)-group. Indeed, \(A_4\) admits an exact factorization through \(V_4\cong C_2\times C_2\) and \(C_3\), and therefore, by \cite[Theorem~3.3]{Smoktunowicz_2018}, there exists a skew brace \(B\) such that $(B,+)\cong A_4$ and $(B,\cdot)\cong V_4\times C_3\cong C_2\times C_2\times C_3$. Since \(A_4\) is not a \(Z\)-group, this shows that the converse implication does not hold.
\end{Example}

\subsection{Proof of Theorem \ref{Th: B ssolv}}

Before proving Theorem \ref{Th: B ssolv} in full generality we first deal with the case in which the skew brace has order a power of $2$:

\begin{Lemma} \label{case p=2}
Let \(B\) be a finite skew brace such that
\[
(B,+)\cong C_{2^k}
\]
for some \(k\geq 1\). Then \(B\) is supersolvable.  
\end{Lemma}

\begin{proof}
We argue by induction on \(k\). If \(k=1\), then \(|B|=2\), and the conclusion is clear.
Assume now \(k\geq 2\), and suppose that the result holds for all skew braces whose additive group is cyclic of order \(2^{k-1}\).
Since \((B,+)\cong C_{2^{k}}\), it follows from the corollary on p.~319 of \cite{Rump2018ClassificationOC} that \(\mathrm{Soc}(B)\neq 1\). Let \(I\) be the unique subgroup of \((\mathrm{Soc}(B),+)\) of order \(2\). Then \(I\) is characteristic in \((\mathrm{Soc}(B),+)\) and so $I$ is characteristic in $(B,+)$.  On the other hand, the additive and multiplicative operations coincide on \(\mathrm{Soc}(B)\), because \(\mathrm{Soc}(B)\subseteq \ker\lambda\). Thus $(\mathrm{Soc}(B),+)= (\mathrm{Soc}(B),\cdot)$ and so \(I\) is also characteristic in \((\mathrm{Soc}(B),\cdot)\). Since \((\mathrm{Soc}(B),\cdot)\trianglelefteq (B,\cdot)\), it follows that $(I,\cdot)\trianglelefteq (B,\cdot)$. Therefore \(I\) is an ideal of \(B\) of order \(2\). By induction $B/I$ is supersolvable and so $B$ is supersolvable.
\end{proof}

We are ready to prove Theorem \ref{Th: B ssolv}.

\begin{proof}[Proof of Theorem \ref{Th: B ssolv}]
Suppose the theorem is false, and let \(B\) be a counterexample of minimal order. Let $r$ be the smallest prime dividing the order of $B$. We first show that
\[
r \mid [(B,+):(B,+)'].
\]
Indeed, since \((B,+)\) is a \(Z\)-group, reasoning as in Theorem \ref{th: generalizedRump} we may write
\[
(B,+)=A\rtimes C_{r^k}
\]
for some subgroup \(A\leq (B,+)\) such that \(r\nmid |A|\). Since
\[
(B,+)/A \cong C_{r^k},
\]
we deduce that
\[
(B,+)'\leq A.
\]
Thus \(r\nmid |(B,+)'|\), and therefore
\[
r \mid [(B,+):(B,+)'].
\]
Therefore, by \cite[Lemma 2.1]{GorshkovNasybullov2021} there exists a characteristic subgroup \(K\) of \((B,+)\) such that
\[
(B,+)/K \cong (C_r)^m
\]
for some \(m\). Since \((B,+)\) is a \(Z\)-group, also \((B,+)/K\) is a \(Z\)-group. Hence \(m=1\), and so
\[
[(B,+):K]=r.
\]
Since \(K\) is characteristic in \((B,+)\)
\[
(K,\cdot)\leq (B,\cdot).
\]
Moreover,
\[
|(B,\cdot):(K,\cdot)|=|(B,+):K|=r,
\]
Since $r$ is the smallest prime dividing \(|B|\) we deduce that  \((K,\cdot)\trianglelefteq (B,\cdot)\). Thus \(K\) is an ideal of \(B\). Since \(B\) is a minimal counterexample, the skew brace $(K,+,\cdot)$ is supersolvable and so  \((K,\cdot)\) is supersolvable. Moreover, by Theorem~\ref{th: generalizedRump}, the group \((K,\cdot)\) is almost Sylow cyclic. Let \(p\) be the largest prime dividing \(|B|\). If $p$ is even, then $p=r=2$ and the Theorem follows from Lemma \ref{case p=2}. Thus assume that $p$ is odd. Since \((B,+)\) is supersolvable, by \cite[Theorem 9.1]{Huppert1967FiniteGroupsI} its Sylow \(p\)-subgroup is unique; denote it by \(P\). Hence \(P\) is characteristic in \((B,+)\), and therefore
\[
(P,\cdot)\leq (B,\cdot).
\]
and $(P,\cdot)$ is a Sylow $p$-subgroup of $(B,\cdot)$. Since \(p\neq r\) and \([(B,+):K]=r\), we also have
\[
P\leq K,
\]
and therefore
\[
(P,\cdot)\leq (K,\cdot).
\]
Now \((P,\cdot)\) is a Sylow \(p\)-subgroup of \((K,\cdot)\). Since \((K,\cdot)\) is supersolvable and \(p\) is the largest prime dividing \(|K|\), again by \cite[Theorem 9.1]{Huppert1967FiniteGroupsI} \((P,\cdot)\) is characteristic in \((K,\cdot)\). As \((K,\cdot)\trianglelefteq (B,\cdot)\), we conclude that \((P,\cdot)\) is normal in \((B,\cdot)\) and thus characteristic. Take \(I\leq (P,+)\) of order \(p\). Since \((P,+)\) is cyclic, \(I\) is characteristic in \((P,+)\). As \(P\) is characteristic in \((B,+)\) we have
\[
(I,\cdot)\leq (B,\cdot).
\]
By order considerations,
\[
(I,\cdot)\leq (P,\cdot).
\]
Since $p$ is odd the group \((P,\cdot)\) is cyclic. Thus $(P,\cdot)$ has a unique subgroup of order \(p\), say \(J\). 
 Therefore
\[
I=J.
\]
It follows that \((I,\cdot)\) is normal in \((B,\cdot)\). Hence \(I\) is an ideal of \(B\) of order \(p\). By induction $B/I$ is supersolvable. Since $I$ has order $p$ we conclude that $B$ is supersolvable, a contradiction.
\end{proof}

As a corollary, we obtain a generalization of \cite[Proposition 13]{Rump2018ClassificationOC}.

\begin{Corollary}
    Let $B$ be a finite skew brace such that $(B,+)$ is a $Z$-group. Then the Sylow $2$-subgroup of $(B,\cdot)$ admits a normal complement.
\end{Corollary}

\begin{proof}
    Since $(B,+)$ is a $Z$-group, by Theorem \ref{Th: B ssolv} $B$ is supersolvable. Thus $(B,\cdot)$ is supersolvable and by \cite[Theorem 9.1]{Huppert1967FiniteGroupsI} the Sylow $2$-subgroup of $(B,\cdot)$ admit a normal complement.
\end{proof}

\subsection{Proof of Theorem \ref{Th: Multiplicative-Z-group}}

We begin by recalling a standard fact from finite group theory that will be needed in the proof of Theorem \ref{Th: Multiplicative-Z-group}. For completeness, we provide a short proof. Recall that, if \(G\) is a finite group, then the Fitting subgroup \(F(G)\) is the largest normal nilpotent subgroup of \(G\):

\begin{Lemma}\label{lem:derived-in-fitt}
Let \(G\) be a finite solvable group. If \(F(G)\) is cyclic, then $G'\leq F(G)$.
\end{Lemma}

\begin{proof}
Since \(F(G)\) is cyclic, it is abelian, and therefore $C_G(F(G))\supseteq F(G)$. On the other hand, by \cite[Theorem 4.2]{Huppert1967FiniteGroupsI} $C_G(F(G))\leq F(G)$. Hence $C_G(F(G))=F(G)$. Therefore the conjugation action of \(G\) on \(F(G)\) induces an injective homomorphism
\[
G/F(G)\longrightarrow \Aut(F(G)).
\]
Since \(F(G)\) is cyclic, the group \(\Aut(F(G))\) is abelian. It follows that \(G/F(G)\) is abelian, and thus $G'\leq F(G)$.
\end{proof}

We are now in a position to prove Theorem~\ref{Th: Multiplicative-Z-group}.

\begin{proof}[Proof of Theorem \ref{Th: Multiplicative-Z-group}]
By Theorem~\ref{th: generalizedRump}, it is enough to prove that if \((B,\cdot)\) is a \(Z\)-group, then \((B,+)\) is also a \(Z\)-group. Assume therefore that \((B,\cdot)\) is a \(Z\)-group, and suppose by contradiction that \((B,+)\) is not a \(Z\)-group. Let \(B\) be a counterexample of minimal order. Since \(|B|\) is odd, the Feit--Thompson theorem (see the main Theorem of \cite{FeitThompson1963}) implies that \((B,+)\) is solvable. Hence, the Fitting subgroup of $(B,+)$ is not trivial: \[ F(B,+)\neq 1. \] Let \(p\) be a prime divisor of \(|F(B,+)|\), and let \(P\in \mathrm{Syl}_p(F(B,+))\). Since \(F(B,+)\) is nilpotent, \(P\) is characteristic in \(F(B,+)\), hence characteristic in \((B,+)\). Therefore \((P,+,\cdot)\) is a subbrace of \(B\). Moreover, \((P,\cdot)\) is a \(p\)-subgroup of \((B,\cdot)\). Since \((B,\cdot)\) is a \(Z\)-group, all its Sylow subgroups are cyclic, and thus \((P,\cdot)\) is cyclic. It then follows from \cite[Theorem~1.2]{damele2026finitegroupspropercharacteristic} that \((P,+)\) is cyclic. Thus every Sylow subgroup of \(F(B,+)\) is cyclic. Since \(F(B,+)\) is nilpotent, it follows that \(F(B,+)\) itself is cyclic. Thus, by Lemma \ref{lem:derived-in-fitt} $(B,+)'$ is cyclic. In particular, by \cite[Lemma 2.17]{Alshorm2022} \((B,+)\) is supersolvable. Now choose the smallest prime \(p\) that divides $|B|$ and write \[ |B|=mp^a,\qquad (m,p)=1, \] Since \((B,\cdot)\) is a \(Z\)-group, it follows from the argument in the proof of Theorem \ref{th: generalizedRump} that \((B,\cdot)\) is supersolvable. Therefore, by \cite[Theorem 9.1]{Huppert1967FiniteGroupsI}, we have
\[
(B,\cdot)=A\rtimes C_{p^a},
\]
where \(A\) denotes the normal Hall \(p'\)-subgroup of \((B,\cdot)\) of order \(m\). In particular, \(A\) is uniquely determined by \cite[Main Theorem 1.7]{Huppert1967FiniteGroupsI}. Since \((B,+)\) is supersolvable, let \(N\) be a normal Hall \(p'\)-subgroup of \((B,+)\). By \cite[Main Theorem 1.7]{Huppert1967FiniteGroupsI} $N$ is a characteristic subgroup of $(B,+)$. Thus $(N,+,\cdot)$ is a subbrace of $B$. Since $(N,\cdot)$ is a $p'$-Hall of $(B,\cdot)$ we deduce that $N=A$. In particular \(N\) is an ideal of \(B\). By minimality of \(B\), the skew brace \(N\) satisfies that \((N,+)\) is a \(Z\)-group, since \((N,\cdot)=(A,\cdot)\) is a \(Z\)-group. Consider now the quotient skew brace \(B/N\). Since \(N\) is an ideal, we have \[ (B/N,\cdot)\cong (B,\cdot)/(N,\cdot)\cong C_{p^a}. \] By \cite[Theorem~1.2]{damele2026finitegroupspropercharacteristic}, the additive group \((B/N,+)\) is cyclic. Since \(N\) is a normal Hall \(p'\)-subgroup of \((B,+)\), we may write \[ (B,+)=N\rtimes Q, \] where \(Q\in \mathrm{Syl}_p(B,+)\). As \[ (B,+)/N\cong Q, \] and \((B,+)/N\) is cyclic, it follows that \(Q\) is cyclic. A contradiction.

\end{proof}

\begin{Example}
The assumption that \(B\) has odd order in Theorem~\ref{Th: Multiplicative-Z-group} is necessary. Indeed, by \cite[Example~2.11]{acri2019skewleft}, there exists a skew brace \(B\) such that $(B,\cdot)\cong C_8$ and $(B,+)\cong D_8$. In particular, \((B,\cdot)\) is a \(Z\)-group, whereas \((B,+)\) is not.
\end{Example}

As a corollary we obtain:

\begin{Corollary}
    Let $B$ be a finite skew brace of odd order such that either $(B,+)$ or $(B,\cdot)$ is a $Z$-group. Then $B$ is supersolvable.
\end{Corollary}

\begin{proof}
    It follows from Theorem \ref{Th: Multiplicative-Z-group} and Theorem \ref{Th: B ssolv}.
\end{proof}

\bibliographystyle{plain}
\bibliography{bibgroup}
\end{document}